\documentclass[12pt]{amsart}
\usepackage{amsthm, amsmath, amssymb, tikz-cd, fullpage, hyperref, cleveref, verbatim, xcolor,enumitem}

\newcommand{\mbn}{\mathbb{N}}

\newcommand{\mfm}{\mathfrak{m}}
\newcommand{\mfc}{\mathfrak{c}}
\newcommand{\mfn}{\mathfrak{n}}
\newcommand{\mfq}{\mathfrak{q}}
\newcommand{\mfp}{\mathfrak{p}}

\newcommand{\fbp}[1]{\left[ #1\right]}
\newcommand{\im}{\operatorname{im}}
\newcommand{\fte}{\operatorname{Fte}}
\newcommand{\spec}{\operatorname{Spec}}

\newcommand{\Max}{\operatorname{Max}}
\newcommand{\Min}{\operatorname{Min}}

\theoremstyle{definition}
\newtheorem{sthm}{Theorem}[section]
\newtheorem{sdef}[sthm]{Definition}
\newtheorem{slem}[sthm]{Lemma}
\newtheorem{sprop}[sthm]{Proposition}

\newtheorem{scor}[sthm]{Corollary}
\newtheorem{srmk}[sthm]{Remark}
\newtheorem{sex}[sthm]{Example}

\newtheorem*{mthm}{Theorem}
\renewcommand{\phi}{\varphi}

\hypersetup{
    colorlinks,
    linkcolor={red!50!black},
    citecolor={blue!50!black},
    urlcolor={blue!80!black}
}
\begin{document}
\title{Counting geometric branches via the Frobenius map and $F$-nilpotent singularities}

\author[Hailong Dao]{Hailong Dao}
\address{Department of Mathematics, University of Kansas, 405 Snow Hall, 1460 Jayhawk Bvld, Lawrence, KS 66045, USA}
\email{hdao@ku.edu}
\author[Kyle Maddox]{Kyle Maddox}
\address{Department of Mathematical Sciences, University of Arkansas, SCEN 309, 850 West Dickson Street, Fayetteville, Ar, 72701, USA}
\email{kmaddox@uark.edu}
\author[Vaibhav Pandey]{Vaibhav Pandey}
\address{Department of Mathematics, Purdue University, 150 N University St., West Lafayette, IN~47907, USA}
\email{pandey94@purdue.edu}

\begin{abstract}
We give an explicit formula to count the number of geometric branches of a curve in positive characteristic using the theory of tight closure. This formula readily shows that the property of having a single geometric branch characterizes $F$-nilpotent curves. Further, we show that a reduced, local $F$-nilpotent ring has a single geometric branch; in particular, it is a domain. Finally, we study inequalities of Frobenius test exponents along purely inseparable ring extensions with applications to $F$-nilpotent affine semigroup rings.   
\end{abstract}


\maketitle

\section{Introduction}

The number of geometric branches of a local ring $(R,\mfm)$ is the number of minimal primes of its strict henselization. Studying the strict henselization of a ring is important to understand its geometry. We can view the strict henselization as the most complete geometric realization of a ring, where no additional elements can arise as roots of monic polynomials or from the separable closure of the residue field $R/\mfm$. We recall some basic facts about the strict henselization and geometric branches of a local ring in \Cref{subsection:sh and geo branches}.

In this paper, we give a formula to count the number of geometric branches of an excellent, reduced, local ring of dimension one in positive prime characteristic. 

\begin{mthm}[\Cref{thm: counting branches main thm}]
Let $(R,\mfm,k)$ be an excellent, reduced local ring of dimension one and of prime characteristic $p>0$. Further, let $(S,\mfn,\ell)$ be its weak normalization inside its total ring of quotients. Let $b(R)$ be the number of geometric branches of $R$. Then \[
\dim_k 0^*_{H^1_\mfm(R)}/0^F_{H^1_\mfm(R)} = [\ell:k](b(R)-1).
\] In particular, if the field $k$ is perfect, we have $b(R)=\dim_k 0^*_{H^1_\mfm(R)}/0^F_{H^1_\mfm(R)} + 1.$
\end{mthm}

In \cite{SW08}, Singh and Walther give a formula to count the number of connected components of the punctured spectrum of the strict henselization of a complete local ring $(R,\mfm)$ with algebraically closed coefficient field using the \textit{semi-stable part} of the Frobenius action on $H^1_\mfm(R)$; see \Cref{rmk: SW reconciliation}. Our work extends theirs in dimension one by removing the hypotheses that $R$ be complete and that the residue field be algebraically closed. There are also computational advantages to our results since an $R/\mfm$-vector space basis of $0^*_{H^1_\mfm(R)}/0^F_{H^1_\mfm(R)}$ is readily available in many situations, including when $R$ is standard graded over the field $R/\mfm$. We compute $b(R)$ in several examples using this technique in \Cref{sec: Frobenius and geometric branches}.

A major aim of this paper is to understand rings with a single geometric branch. In the light of the above theorem, this naturally leads us to study $F$-nilpotent rings---a recently introduced singularity type in prime characteristic. Defined by Blickle and Bondu in \cite{BB} under the name ``close to $F$-rational", a local ring $(R,\mfm)$ of dimension $d$ is \textit{$F$-nilpotent} if, for each $j<d$, the canonical Frobenius action on $H^j_\mfm(R)$ is nilpotent, and the tight closure of the zero submodule in $H^d_\mfm(R)$ is also nilpotent, that is, $0^*_{H^d_\mfm(R)} = 0^F_{H^d_\mfm(R)}$.  

In \cite{ST17}, Srinivas and Takagi define a ring of characteristic zero to be of \textit{$F$-nilpotent type} if almost all of its mod $p$ reductions are $F$-nilpotent. They give a characterization of two-dimensional normal rings of $F$-nilpotent type over the complex numbers in terms of their divisor class groups. They also give a characterization of three-dimensional graded normal rings of $F$-nilpotent type over the complex numbers in terms of the divisor class groups and Brauer groups (cf. \cite[Theorems 4.1,4.2]{ST17}). In this paper, we show the following.

\begin{mthm}[\Cref{thm: F-nilpotent implies geo unibranched}, \Cref{col:main corollary}]
Suppose $R$ is an excellent, reduced ring of prime characteristic $p>0$. Then, if $R$ is $F$-nilpotent, the normalization map $R_\mfm\rightarrow \overline{R_\mfm}$ is purely inseparable for each maximal ideal $\mfm$ of $R$ so that $R_\mfm$ is geometrically unibranched. In particular, reduced, excellent $F$-nilpotent local rings are domains. 

Furthermore, if $\dim R = 1$, then $R$ is $F$-nilpotent if and only if $R_\mfm$ is geometrically unibranched for each maximal ideal $\mfm$ of $R$.
\end{mthm}

A key insight of this paper is that the number of branches of a local ring in positive characteristic can be counted by studying its weak normalization. The weak normalization of a reduced ring encapsulates the purely inseparable part of its normalization. In \cite{Sch09}, Schwede showed that an $F$-injective ring which admits a dualizing complex (a very mild requirement) must be weakly normal. In the course of proving the above theorem, we show that $F$-nilpotent rings exhibit a ``dual" property to $F$-injective rings, in that the weak normalization of an $F$-nilpotent ring must itself be normal. Since an $F$-rational local ring is precisely one which is both $F$-injective and $F$-nilpotent, our result, together with that of Schwede, provides, perhaps amusingly, a novel proof of the well-known fact that an $F$-rational local ring is a normal domain; see Remark \ref{rmk: dual to schwede}.

As a final application of our techniques, we study the computational aspects of trivializing the Frobenius closure of parameter ideals of a local ring using its weak normalization. The Frobenius test exponent $\fte R$ of a local ring $(R,\mfm)$ is the smallest $e$ (if one exists) such that $(\mfq^F)^{\fbp{p^e}}=\mfq^{\fbp{p^e}}$ for all parameter ideals $\mfq$ of $R$. In \cite{KS05}, Katzman-Sharp showed that Cohen-Macaulay local rings have finite Frobenius test exponents and in \cite{Quy}, Quy showed that $F$-nilpotent local rings also have this property. We show that to determine whether an (excellent, reduced) local ring has finite Frobenius test exponent, it suffices to determine whether its weak normalization has  finite Frobenius test exponent.

\begin{mthm}[\Cref{cor: weak normalization finite fte}]
Let $R$ be an excellent, reduced local ring of prime characteristic and write $S$ for its weak normalization. Then, $\fte R$ is finite if and only if $\fte S$ is finite.
\end{mthm}

The following theorem demonstrates that the usually intractable calculations involved in computing the tight closure of ideals are much easier in $F$-nilpotent affine semigroup rings, and furthermore, that the pure inseparability of the normalization map characterizes $F$-nilpotent affine semigroup rings in any dimension.

\begin{mthm}[\Cref{cor: f-nilp affine semigroup ring}]
Suppose $R$ is a locally excellent domain and its integral closure $\overline{R}$ is $F$-regular (for instance, if $R$ is an affine semigroup ring defined over a field $k$ of prime characteristic $p>0$). Then, $R$ is $F$-nilpotent if and only if $R\rightarrow \overline{R}$ is purely inseparable. Further, if this is the case, then  $I^F=I^*$ for all ideals $I$ of $R$ and $\fte I\le e_0$, where $e_0$ is the pure inseparability index of $R\rightarrow \overline{R}$. 
\end{mthm}

\section{Preliminaries}

All rings considered in this paper are commutative with identity and Noetherian; further, we often assume that our rings are reduced and excellent. A reduced, excellent local ring $(R,\mfm)$ is \textit{analytically unramified}, so that its $\mfm$-adic completion $\widehat{R}$ is reduced. This is equivalent to the property that the integral closure $\overline{R}$ of $R$ in its total ring of quotients is a finite $R$-module. In many of the theorems that follow, the conditions reduced and excellent can be relaxed to analytically unramified. Finally, we will almost universally assume that our rings are of positive prime characteristic. 

\subsection{Strict henselization and geometric branches} \label{subsection:sh and geo branches}

Throughout this subsection, let $(R,\mfm)$ be a local ring. We mention several important facts about the (strict) henselization which we will utilize later in the paper.

A ring $R$ is said to be \textbf{henselian} if it satisfies the conclusions of Hensel's lemma and \textbf{strictly henselian} if it is henselian and the residue field $R/\mfm$ is separably closed. The henselization $R^h$ of $R$ is the unique ring satisfying a universal mapping property with respect to maps from $R$ to any henselian ring. In particular, $R^h$ is obtained from $R$ by taking the direct limit of all local extensions $R\rightarrow R'$ which are \textit{\'etale} and induce an isomorphism on residue fields. The strict henselization is similar---it is constructed by taking the limit of all local extensions $(R,\mfm,k)\rightarrow (S,\mfn,\ell)$ such that $\ell$ is a subfield of $k^{\text{sep}}$ and the composition of the inclusions $k\rightarrow \ell \rightarrow k^{\text{sep}}$ agrees with the inclusion $k\rightarrow k^{\text{sep}}$. 

We will use the following well-known property of the (strict) henselization.

\begin{sthm}\label{thm: properties of henselization}
Write $k=R/\mfm$, and fix a separable closure $k^{\text{sep}}$ of $k$. Then, $(R^h,\mfm^h)$ and $(R^{sh},\mfm^{sh})$ are local rings, $R\rightarrow R^h\rightarrow R^{sh}$ is a sequence of faithfully flat unramified maps, and $R^h/\mfm^h \simeq k$ and $R^{sh}/\mfm^{sh} \simeq k^{\text{sep}}$.
\end{sthm}

Since $R\rightarrow R^h \rightarrow R^{sh}$ is a sequence of faithfully flat maps, the induced maps on spectra are surjective. In particular, we must have that $|\Min R|\le |\Min R^h| \le |\Min R^{sh}|$. Thus, if either $R^h$ or $R^{sh}$ has a unique minimal prime, so does $R$.

\begin{sdef}
Let $R$ be a ring and let $\mfm$ be a maximal ideal of $R$. The \textbf{number of (geometric) branches of $R$ at $\mfm$} is the number of minimal primes of the (strict) henselization of the local ring $R_\mfm$. If $(R,\mfm)$ is local, we denote the number of geometric branches of $R$ by $b(R)$, that is, $b(R)=|\Min R^{sh}|$. If $R$ has a single (geometric) branch, then we say that $R$ is \textbf{(geometrically) unibranched}.
\end{sdef}

The notions of branches and geometric branches of a ring can also be understood by studying the normalization of the ring. We next recall that the geometric branches of a local ring can be counted by the sums of the separable degrees of certain extension fields of the residue fields arising from the normalization map; see \cite[Tag~0C37(5)]{stacks}. 

\begin{srmk}\label{rmk: geometric branches counted by separable degree}
Let $(R,\mfm)$ be a local ring with $k=R/\mfm$ and let $\overline{R}$ be its normalization, a semi-local ring. By \cite[Tag~0C24]{stacks}, the maximal ideals of $\overline{R}$ correspond bijectively with the minimal primes of $R^{h}$ and the minimal primes of the completion  $\widehat{R}$ of $R$ at $\mfm$. Thus, $R$ is unibranched if and only if its normalization $\overline{R}$ is a local ring. 

Next, write $\Max \overline{R} = \{\mathfrak{M}_1,\ldots,\mathfrak{M}_b\}$, and $K_i = \overline{R}/\mathfrak{M}_i$. Then, \[b(R) = \sum_{i=1}^b [K_i:k]_{\text{sep}},\] where $[K_i:k]_{\text{sep}}$ is the separable degree of the extension $k\subset K_i$. Thus, $R$ is geometrically unibranched ($b(R) = 1$) if and only if $\overline{R} = (\overline{R}, \mathfrak{M}, K)$ is a local ring and the field extension $k \rightarrow K$ is \textit{purely inseparable}, that is, for each $x$ in $K$, there exists some positive integer $e$ such that $x^{p^e}$ lies in $k$ where $p>0$ is the characteristic of $k$. 
\end{srmk}

We provide an example to illustrate the difference between the number of branches and geometric branches of a ring. 
 
\begin{sex} 
Let $R$ the ring $\mathbb{F}_3[x,y]/(x^2+y^2)$ localized at the ideal $\mathbf{m} = (x,y)$ where $\mathbb{F}_3$ is the field with three elements. Note that the completion of $R$ at $\mathbf{m}$ is $\mathbb{F}_3[|x,y|]/(x^2+y^2)$, which is a domain. Therefore $R$ has a single branch. Alternatively, notice that $t = y/x$ lies in the normalization $\overline{R}$ of $R$. Further, $R[t]$ is the ring \[ \left( \frac{\mathbb{F}_3[t]}{(t^2+1)}[x]\right)_{\mathbf{m}}\simeq (\mathbb{F}_9[x])_{(x)}.\] Since $R[t]$ is a normal domain, it must be equal to $\overline{R}$. As the ring $\overline{R}$ is local, we again see that $R$ has a single branch.

Notice however that $R$ has two geometric branches, that is, $b(R) = 2$. This is because $R^{sh}$ is the ring $\mathbb{F}_3^{\text{sep}}[x,y]/(x+iy)(x-iy)$ localized at the ideal $(x,y)$, which has two minimal primes; here $i$ is a root of the separable polynomial $t^2+1$ over $\mathbb{F}_3[t]$. Alternatively, $\overline{R}$ is the ring $(\mathbb{F}_9[x])_{(x)}$ and $[\mathbb{F}_9 : \mathbb{F}_3]_{\text{sep}} = [\mathbb{F}_9 : \mathbb{F}_3] =2$ also confirms that $R$ has two geometric branches by \Cref{rmk: geometric branches counted by separable degree}.
\end{sex}

\subsection{Submodule closures in prime characteristic}

Throughout this subsection, let $R$ be a ring of prime characteristic $p>0$. The \textbf{Frobenius map} $F:R\rightarrow R$ is defined by $F(r)=r^p$, and is a ring endomorphism since $R$ is of characteristic $p$. We may denote denote the target of $F$ as $F_*(R)$, which we view as an $R$-module via $rF_*(s) = F_*(r^ps)$. 

\begin{sdef}\label{def: tight and frob closure}
Let $N\subset M$ be $R$-modules. The \textbf{Frobenius closure of $N$ in $M$}, denoted $N^F_M$, is the $R$-submodule of elements which vanish under the composition \begin{center}
    \begin{tikzcd}
        M\arrow{r}{\pi} & M/N \arrow{rr}{\operatorname{id}\otimes_R F^e} && M/N\otimes_R F^e_*(R)
    \end{tikzcd}
\end{center} for some $e \ge 0$. Similarly, if $R^\circ = \{c \in R \mid c\not\in\mfp \text{ for any } \mfp\in\Min R\}$, then the \textbf{tight closure of $N$ in $M$}, denoted $N^*_M$, is the $R$-submodule of elements which for some $c\in R^\circ$ vanish under the composition \begin{center}
    \begin{tikzcd}
        M\arrow{r}{\pi} & M/N \arrow{rr}{\operatorname{id}\otimes_R F^e} && M/N\otimes_R F^e_*(R) \arrow{rr}{\operatorname{id}\otimes_R \cdot F^e_*(c)} && M/N\otimes_R F^e_*(R) 
    \end{tikzcd}
\end{center} for all $e \gg 0$.
\end{sdef}

For any $R$-module $N$ of $M$, we have $N\subset N^F_M\subset N^*_M$. Of special interest is the case $M=R$ and $N=I$ is an ideal in $R$, where the definitions above agree with the usual Frobenius and tight closure of ideals. An interesting and largely open problem is to find methods to compute the tight and Frobenius closure of an ideal in a given ring. For now, we will focus on Frobenius closure. 

\begin{sdef}\label{def: Fte of I}
Let $I$ be an ideal of $R$. Since the ideal $I^F$ is finitely generated, there must be a positive integer $e$ such that for any $x \in I^F$, we have $x^{p^e} \in I^{\fbp{p^e}}$. Call the smallest such $e$ the \textbf{Frobenius test exponent of $I$}, written $\fte I$. 
\end{sdef}

Knowing $\fte I$ (or even an upper bound on $\fte I$) is desirable to compute $I^F$ since we can check whether $x \in I^F$ using a single equation instead of \textit{a priori} needing to check infinitely many. Even more useful in computing Frobenius closure in a ring $R$ would be knowing an upper bound on $\fte I$ over all ideals $I$ in $R$. Unfortunately, Brenner showed in \cite{Bre06} that no uniform upper bound on $\fte I$ can exist over all ideals $I$ in $R$, even if $R$ is a standard graded normal domain of dimension two. 

In contrast, Katzman-Sharp showed in \cite{KS05} that if $(R,\mfm)$ is a Cohen-Macaulay local ring, there is a uniform upper bound on $\fte \mfq$ over all ideals $\mfq$ in $R$ generated by a (partial) system of parameters. These ideals are called parameter ideals.

\begin{sdef}\label{def: Fte of R}
Let $(R,\mfm)$ be a local ring. The \textbf{Frobenius test exponent} (for parameter ideals) of $R$, written $\fte R$, is \[
\fte R = \sup \{ \fte \mfq \mid \mfq \subset R \text{ is a parameter ideal}\}\in \mbn\cup\{\infty\}.
\]
\end{sdef}

In particular, the result of Katzman-Sharp states that a Cohen-Macaulay local ring has a finite Frobenius test exponent. For a survey of other cases where the Frobenius test exponent is known to be finite, see \cite{Mad19}. Our techniques in \Cref{sec: Fte and pi exts} compute bounds on Frobenius test exponents using purely inseparable ring extensions.

We conclude this subsection with a useful lemma regarding the tight and Frobenius closure of a general linear form in a one-dimensional graded ring.

\begin{slem}\label{lem: tight and frob closure of reduction of m}
Let $R$ be a reduced ring of dimension one with homogeneous maximal ideal $\mfm$ and standard graded over an infinite field $k$. Further, suppose $x\in [R]_1$ is a reduction of the homogeneous maximal ideal\footnote{Note that such an $x$ exists since a general $k$-linear combination of the generators of $[R]_1$ is a reduction of $\mfm$ as $k$ is infinite.} $\mfm$, with $\mfm^{N+1}=x\mfm^N$. Then, the tight and Frobenius closure of $(x^n)$ for $n \ge N$ are as follows:
\begin{enumerate}[label=(\alph*)]
    \item $(x^n)+\mfm^{n+1} \subset (x^n)^F$, and if $k$ is perfect, then equality is attained.
    \item $(x^n)^* = \mfm^n$.
\end{enumerate}
Finally, the above equalities also hold in the local ring $R_\mfm$.
\end{slem}

\begin{proof}
Let $a \in (x^n)+\mfm^{n+1}$. Then, $a = bx^n + w$, where $w \in \mfm^{n+1}$, and $a^{p^e} = b^{p^e}x^{np^e} + w^{p^e}$ for all $e \in \mbn$. But for $e \gg 0$, $w^{p^e} \in \mfm^{(n+1)p^e} = x^{(n+1)p^e-N}\mfm^N$, so that $a^{p^e}=sx^{np^e}$ for some $s \in R$, that is, $a\in (x^n)^F$.

Now suppose $k$ is perfect, and let $y \in (x^n)^F$. Since $(x^n)$ is a homogeneous ideal, so is $(x^n)^F$, thus it suffices to assume that $y$ is homogeneous. Further, we may assume $\deg y = n$, as otherwise $y \in \mfm^{n+1}$. Then, for all $e \ge \fte (x^n)$ there is an $r \in R$ with $y^{p^e} = rx^{np^e}$, counting degrees on both sides, we must have $\deg r = 0$, that is, $r \in k$. Since the field $k$ is perfect, there is an $s \in k$ with $s^{p^e} = r$, and so $y^{p^e} = (sx^n)^{p^e}$. As $R$ is reduced, we get $y \in (x^n)$, concluding the proof of (a).
    
Since the ideal $(x^n)$ is principal, its tight closure $(x^n)^*$ equals its integral closure $\overline{(x^n)}$. Since $(x^n)$ is a reduction of $\mfm^n$, we must have that $\mfm^n$ is contained in $\overline{(x^n)}$. For the reverse containment, see \cite[Proposition~2.1]{Smi2} for a general statement concerning lower bounds of the degree of an element contained in the tight closure of a homogeneous ideal.

Now we consider the containments above in the local ring $R_\mfm$. The result in this case is a simple consequence of the fact that the ideals above all localize appropriately. In particular, write $(R',\mfn)=(R_\mfm,\mfm R_\mfm)$ and $z$ for the image of $x$ in $R'$. First, note that the ideal equation $x\mfm^N = \mfm^{N+1}$ localizes to the ideal equation $z\mfn^N=\mfn^{N+1}$ so that $z$ continues to be a reduction of $\mfn$ in $R'$, and hence, a parameter of $R'$. Then, for any $n\ge N$, we have
\begin{itemize}
    \item $(x^n)^F R' = (z^n)^F$,
    \item $((x^n)+\mfm^{n+1})R' = (z^n)+\mfn^{n+1}$,
    \item $(x^n)^*R' = \overline{(x^n)}R' = \overline{(z^n)} = (z^n)^*$, and\footnote{Tight closure does not localize in general, but for principal ideals it agrees with the integral closure which does localize.}
    \item $\mfm^n R' = \mfn^n$.
\end{itemize}
This shows the same results hold in the local ring $R_\mfm$.
\end{proof}

\subsection{Weak normalization and purely inseparable extensions}

In this subsection, all rings considered will be of prime characteristic $p>0$. Recall that a field extension $k\rightarrow \ell$ is \textbf{purely inseparable} if $k$ has characteristic $0$ or if $k$ has characteristic $p>0$ and every element $\lambda$ of $\ell$ satisfies an equation of the form $\lambda^{p^e}=x$ for some $x \in k$ and positive integer $e$. One can similarly define the notion of purely inseparable ring extensions as below.

\begin{sdef}
A ring extension (that is, an injective homomorphism) $\phi: R\rightarrow S$ is \textbf{purely inseparable} if for all $s \in S$ there is a natural number $e$ such that the element $s^{p^e}$ lies in $\phi(R)$. If $\phi$ is a finite map, there must be an $e$ so that the set $F^e(S)$ is contained in $R$; we call the smallest such $e$ the \textbf{pure inseparability index} of the map $\phi$.
\end{sdef}

It is clear that a purely inseparable ring extension induces a purely inseparable map on the residue fields of local rings and on the total quotient rings of reduced rings.

\begin{srmk}\label{rmk: pi ext implies nilp coker}
Note that the pure inseparability index of a purely inseparable extension $\phi:R\rightarrow S$ is the same as the \textit{Hartshorne-Speiser-Lyubeznik number} of the module $R$-module $S/\phi(R)$ endowed with the (nilpotent) Frobenius action $\overline{F}(s+\phi(R)) = s^p+\phi(R)$. See \cite[Section~1]{KS05} for a discussion on Hartshorne-Speiser-Lyubeznik numbers. 
\end{srmk}

\begin{sdef}
The largest purely inseparable extension of a reduced ring $R$ inside its total ring of quotients is called the \textbf{weak normalization} ${^*R}$ of $R$. That is, we have \[
{^*R} = \left\lbrace x \in \overline{R} \mid x^{p^e} \in R \text{ for some } e \in \mbn\right\rbrace .
\]  If $R={^*R}$, then $R$ is said to be \textbf{weakly normal}. 
\end{sdef}

To avoid any confusion with the notation for tight closure, we will write $S$ for the weak normalization of $R$. The weak normalization encapsulates the purely inseparable part of the normalization. In particular, we have a sequence of inclusions $R\rightarrow {S}\rightarrow \overline{R}$ whose composition is the natural inclusion of $R$ into its normalization $\overline{R}$, and the map $R\rightarrow {S}$ is purely inseparable. 

\subsection{Prime characteristic singularities}

In this subsection, we continue to let $R$ be a ring of prime characteristic $p>0$. Singularities in prime characteristic are defined in terms of the behavior of the Frobenius endomorphism of $R$. Kunz famously proved that $R$ is regular if and only if $F:R\rightarrow R$ is flat. Some singularity types are too subtle to be detected by the Frobenius map on $R$---they are studied by the natural action of the Frobenius map on the local cohomology modules of $R$.

\begin{sdef}\label{def: frob action on local coh}
For any ideal $I$ of $R$ and natural number $j$, the ring homomorphism $F:R\rightarrow R$ induces an additive map $F:H^j_I(R)\rightarrow H^j_I(R)$ called a \textbf{Frobenius action}. Further, $F$ is \textbf{$p$-linear}, that is, for each $\xi \in H^j_I(R)$ and $r \in R$, we have $F(r\xi)=r^pF(\xi)$.
\end{sdef}

For a discussion on how $F:H^j_I(R)\rightarrow H^j_I(R)$ is induced from $F:R\rightarrow R$, we direct the reader to \cite[Remark~2.1]{Sha06}. One way to measure the singularity of a local ring $(R,\mfm)$ using $F:H^j_\mfm(R)\rightarrow H^j_\mfm(R)$ is by understanding how much of the local cohomology vanishes under high iterates of $F$. Studied by Srinivas-Takagi, Polstra-Quy, Quy, Kenkel-Maddox-Polstra-Simpson among others, the following singularity types are defined by the property that the local cohomology modules of $R$ are as nilpotent as possible.

\begin{sdef}\label{def: weak and F-nilp}
Let $(R,\mfm)$ be a local ring. We say that $R$ is weakly $F$-nilpotent if for each $0 \le j < \dim R$, $H^j_\mfm(R)$ is nilpotent under $F$. Further, $R$ is \textbf{$F$-nilpotent} if, in addition, $0^*_{H^d_\mfm(R)}=0^F_{H^d_\mfm(R)}$, that is, the largest Frobenius stable submodule of $H^d_\mfm(R)$ is nilpotent. A non-local ring $R$ is \textbf{(weakly) $F$-nilpotent} if $R_\mfm$ is (weakly) $F$-nilpotent for all $\mfm \in \Max R$. 

If $R$ is a non-negatively graded ring over a field with homogeneous maximal ideal $\mfm$, we say that $R$ is (weakly) $F$-nilpotent in the same way as for local rings, replacing the local cohomology modules with the graded local cohomology modules  supported at $\mfm$.
\end{sdef}

The class of $F$-nilpotent rings was introduced by Blickle-Bondu in \cite{BB} (under the name \textit{close to $F$-rational}) and studied further by Srinivas-Takagi in \cite{ST17}. They can be viewed as a weakening of $F$-rational rings---a classical $F$-singularity type. We remind the reader of this definition below.

\begin{sdef}\label{def: F-inj and F-rational}
Let $(R,\mfm)$ be a local ring of dimension $d$. We say that $R$ is \textbf{$F$-injective} if $F:H^j_\mfm(R)\rightarrow H^j_\mfm(R)$ is injective for all $j$, and $R$ is \textbf{$F$-rational} if it is Cohen-Macaulay and $0^*_{H^d_\mfm(R)}=0$, that is, $H^d_{\mfm}(R)$ has no nontrivial Frobenius stable submodules.  
\end{sdef}

Observe that a local ring is $F$-rational if and only if it is both $F$-injective and $F$-nilpotent. Due to strong connections between closure operations on parameter ideals and submodule closures inside the local cohomology modules, some of the singularity types outlined in this subsection enjoy uniformity properties with regards to the tight and Frobenius closures of parameter ideals: 

\begin{srmk}
Let $(R,\mfm)$ be an excellent, equidimensional local ring.
\begin{itemize}
\item $R$ is $F$-rational if and only if $\mfq^*=\mfq$ for all (equivalently for one) parameter ideals $\mfq$ of $R$ (\cite[Theorem~2.6]{Smi97}). 
\item If $R$ is Cohen-Macaulay, then $R$ is $F$-injective if and only if $\mfq^F=\mfq$ for all (equivalently for one) parameter ideals $\mfq$ of $R$ (\cite[Corollary~3.9]{QS17}).
\item  $R$ is $F$-nilpotent if and only if $\mfq^*=\mfq^F$ for all parameter ideals $\mfq$ of $R$ (\cite[Theorem~A]{PQ}).
\end{itemize}
\end{srmk}

Finally, we will need to utilize one more singularity type defined in terms of the triviality of tight closure for all ideals, not just parameter ideals.

\begin{sdef}\label{def: weakly F-regular and F-regular}
A ring $R$ is \textbf{weakly $F$-regular} if $I^*=I$ for all ideals $I$ of $R$, and is \textbf{$F$-regular} if $W^{-1}R$ is weakly $F$-regular for all multiplicative sets $W$ of $R$. 
\end{sdef}

Notably, the $F$-regular condition clearly localizes. 

We will demonstrate in \Cref{sec: Fte and pi exts} that rings which are ``close" to having the singularity types given in this subsection (that is, up to a finite, purely inseparable extension) have similar uniformity properties with respect to ideal closures. We conclude this subsection by demonstrating the ascent and descent of (weakly) $F$-nilpotent singularities along purely inseparable extensions (compare with \cite[Theorem~4.5]{KMPS} and \cite[Lemma~2.14]{MP22}).

\begin{sthm} \label{thm: ascent/descent of F-nilpotence}
Let $R\rightarrow S$ be a finite, purely inseparable ring extension. Then $R$ is (weakly) $F$-nilpotent if and only if $S$ is (weakly) $F$-nilpotent. In particular, a purely inseparable (sub)extension of an $F$-regular ring is $F$-nilpotent.
\end{sthm}

\begin{proof}
A map being finite and purely inseparable localizes, and the (weakly) $F$-nilpotent condition is local. Further, since $R\rightarrow S$ is purely inseparable, for each maximal ideal $\mfm$ of $R$, there is a unique maximal ideal $\mfn$ of $S$ containing $\mfm$, and $\mfm S$ is $\mfn$-primary. So, after localizing, we may assume $(R,\mfm)\rightarrow (S,\mfn)$ is a finite, purely inseparable extension of local rings. Write $d = \dim R = \dim S$. 

We have a short exact sequence $0 \to R \to S \to S/R \to 0$, and $S/R$ is nilpotent under the usual Frobenius action by \Cref{rmk: pi ext implies nilp coker}. By \cite[Theorem~3.5]{MM}, $S$ is weakly $F$-nilpotent if and only if $R$ is. 

Since $R\rightarrow S$ is purely inseparable, $\sqrt{\mfm S}=\mfn$, so the modules $H^d_\mfm(S)$ and $H^d_\mfn(S)$ are the same by the change-of-rings property of local cohomology. We then use the following commutative diagram induced by the natural Frobenius action on local cohomology modules: \begin{center}
\begin{tikzcd}
H^{d-1}_\mfm(S/R) \arrow{r}{\delta}\arrow{d}{\overline{F}}& H^d_\mfm(R) \arrow{r}{\alpha} \arrow{d}{F}& H^d_\mfm(S) \arrow{r}{\beta}\arrow{d}{F}&  H^d_\mfm(S/R) \arrow{r}\arrow{d}{\overline{F}} &0 \\
H^{d-1}_\mfm(S/R) \arrow{r}{\delta}& H^d_\mfm(R) \arrow{r}{\alpha}& H^d_\mfm(S) \arrow{r}{\beta}&  H^d_\mfm(S/R) \arrow{r} &0. 
\end{tikzcd}
\end{center} 

Since $S/R$ is nilpotent under the natural Frobenius map $\overline{F}$, its local cohomology modules $H^j_\mfm(S/R)$ are also nilpotent under the action induced by $\overline{F}$; in particular, if $e_0$ is the pure inseparability. index of $R\rightarrow S$, we have $\overline{F}^{e_0}(H^j_\mfm(S/R))=0$ for all $j$.

Now suppose $R$ is $F$-nilpotent, and let $\xi \in 0^*_{H^d_\mfm(S)}$. Then $\beta(\xi)\in H^d_\mfm(S/R)$, so we have $\overline{F}^{e_0}(\beta(\xi))=\beta(F^{e_0}(\xi))=0$. Thus, $F^{e_0}(\xi)=\alpha(\xi')$ for some $\xi\in H^d_\mfm(R)$. Let $c \in R^\circ$ have $cF^e(\xi)=0$ for all $e \gg 0$. Then $0=cF^{e+e_0}(\xi)=\alpha(cF^e(\xi'))$, so that $cF^e(\xi')\in \ker(\alpha)=\im(\delta)\subset H^d_\mfm(R)$ for all $e \gg 0$. But $F^{e_0}\circ\delta = \delta\circ \overline{F}^{e_0}=0$, so that $c^{p^{e_0}}F^{e+e_0}(\xi'))=0$ for all $e \gg 0$. Since $e_0$ is independent of $e$, this shows that $\xi'\in 0^*_{H^d_\mfm(R)}=0^F_{H^d_\mfm(R)}$. Thus, $F^e(\xi')=0$ for all $e \gg 0$, and thus $F^{e_0+e}(\xi)=\alpha(F^e(\xi))=0$ for all $e \gg 0$, that is, $\xi\in 0^F_{H^d_\mfm(S)}$. 

Finally, suppose $S$ is $F$-nilpotent and that $\xi \in 0^*_{H^d_\mfm(R)}$. Then, there is a $c \in R^\circ$ so that $cF^e(\xi)=0$ for all $e \gg 0$, and consequently $\alpha(cF^e(\xi))=cF^e(\alpha(\xi))=0$ for all $e \gg 0$. Thus $\alpha(\xi)$ is in $0^*_{H^d_\mfm(S)} = 0^F_{H^d_\mfm(S)}$, so that $F^e(\xi)\in \ker(\alpha)=\im(\delta)$. But $\im(\delta)$ is nilpotent under $\overline{F}$ as $S/R$ is, so $F^{e+e'}(\xi)=0$, that is, $\xi \in 0^F_{H^d_\mfm(R)}$. 
\end{proof}

\section{The Frobenius map and geometric branches}\label{sec: Frobenius and geometric branches}
In this section, we assume that all rings considered are of prime characteristic $p>0$. To describe the connection between geometric branches and the Frobenius map, we first consider the case of a single geometric branch. We will give a characterization of the property of having a single geometric branch for rings of dimension one. 

\subsection{Geometric unibranchedness}

Recall from Remark \ref{rmk: geometric branches counted by separable degree} that a reduced local ring is geometrically unibranched if and only if it is a domain and the normalization map is purely inseparable. We demonstrate that this is the case for $F$-nilpotent rings. 

\begin{sthm}\label{thm: F-nilpotent implies geo unibranched}
Let $R$ be a excellent, reduced $F$-nilpotent ring. The normalization map $R_\mfm\rightarrow \overline{R_\mfm}$ is purely inseparable for each maximal ideal $\mfm$ of $R$ so that $R_\mfm$ is geometrically unibranched. In particular, excellent, reduced $F$-nilpotent local rings are domains.
\end{sthm}

\begin{proof}
Assume that $(R,\mfm)$ is local. We first show that $R \rightarrow \overline{R}$ is a purely inseparable map of rings. Let $x/y$ be an element of the total quotient ring of $R$ which is integral over $R$. Then, $y$ is either a unit of $R$ (in which case the remainder of the argument is trivial) or a regular element of $R$, and $x \in \overline{(y)}$, the integral closure of the ideal $(y)$ in $R$. However, by \cite[Corollary~5.8]{HH}, $\overline{(y)}=(y)^*$ as $(y)$ is principal. But since $S$ is an excellent $F$-nilpotent local ring and $y$ is a parameter element, we have $(y)^*=(y)^F$ by \cite[Corollary~5.15]{PQ}.

Thus, there is a natural number $e$ so that $x^{p^e} \in (y)^{\fbp{p^e}}=(y^{p^e})$, that is, $x^{p^e}=ry^{p^e}$ for some $r \in R$. Note that this means $x/y \in S$, the weak normalization of $R$. Since $x/y$ was chosen to be an arbitrary element of $\overline{R}$, we get that $S = \overline{R}$, that is, the weak normalization of $R$ is normal. In particular, $R \rightarrow \overline{R}$ is a purely inseparable map of rings. 

Since a purely inseparable map of rings induces a homeomorphism on spectra, we have that $\overline{R}$ is also a local ring. Let $\mfn$ be the unique maximal ideal of $\overline{R}$ and $K:=\overline{R}/\mfn$. Clearly, the purely inseparable ring map $R \rightarrow \overline{R}$ induces a purely inseparable map of residue fields $k \rightarrow K$, so we have $[K:k]_{\text{sep}} = 1$. Thus, $R$ is geometrically unibranced, as claimed.

Finally, note that since $\overline{R}$ is local, the henselization $R^h$ of $R$ has a unique minimal prime ideal. As $R$ is reduced and $R^h$ is a filtered colimit of \emph{\'etale}, hence smooth $R$-algebras, so $R^h$ is also reduced. It follows that $R^h$ is a domain, and therefore that $R$ is a domain.
\end{proof}

\begin{srmk}\label{rmk: f-nilpotent has prime nilradical}
The above theorem still applies in the case that $R$ is not reduced, since $R$ is $F$-nilpotent if and only if $R/\sqrt{0}$ is where $\sqrt{0}$ denotes the nilradical of $R$. Further, the definition of geometric unibranchedness only depends on $R/\sqrt{0}$. In particular, if $(R,\mfm)$ is a (not necessarily reduced) excellent local ring which is $F$-nilpotent, $R/\sqrt{0}$ is a domain and thus $R$ has a unique minimal prime. 
\end{srmk}

In \cite{Sch09}, Schwede showed under very mild restrictions (the existence of a dualizing complex\footnote{A local ring possesses a dualizing complex if and only if it is a homomorphic image of a finite dimensional local Gorenstein ring; see \cite[Theorem 1.2]{Kaw00}.}) that $F$-injective rings are weakly normal. The proof of the theorem above demonstrates the following analogous property for $F$-nilpotent rings.

\begin{srmk}\label{rmk: dual to schwede}
Let $R$ be an excellent local ring which admits a dualizing complex, and let $S$ be its weak normalization. If $R$ is $F$-injective, then $R$ is reduced. Further, by \cite[Theorem~4.7]{Sch09}, an $F$-injective ring is weakly normal, so the first inclusion in the sequence $R\rightarrow {S} \rightarrow \overline{R}$ is an equality. We have shown above that an $F$-nilpotent ring has a unique minimal prime and the second inclusion $S \rightarrow \overline{R}$ must be equality. Since a ring is $F$-rational if and only if it is both $F$-nilpotent and $F$-injective, this appears to provide a novel proof of the well-known fact that a local $F$-rational ring is a normal domain in the excellent case.
\end{srmk}

\Cref{thm: F-nilpotent implies geo unibranched} shows that we should expect a strong connection between the failure of $F$-nilpotence and the existence of multiple geometric branches. The case of dimension one is already interesting, and we will demonstrate that we can count the number of geometric branches using the module $0^*_{H^1_\mfm(R)}/0^F_{H^1_\mfm(R)}$.

\subsection{Counting geometric branches in dimension one} 
We need the following fact for our main result. 

\begin{slem}\label{lem: m annihilates 0^*/0^F}
Suppose $(R,\mfm)$ is an excellent, reduced local ring of dimension one. Then, the $R$-module $\mfm 0^*_{H^1_\mfm(R)}$ is contained in $0^F_{H^1_\mfm(R)}$ so that $0^*_{H^1_\mfm(R)}/0^F_{H^1_\mfm(R)}$ is an $R/\mfm$-vector space.  
\end{slem}

\begin{proof}
If $R$ is normal, then it is regular, so the result is trivial. Otherwise, the conductor ideal $\mfc$ of $R\rightarrow \overline{R}$ is an $\mfm$-primary ideal of $R$. Let $e$ be the smallest natural number so that $\mfm^{\fbp{p^e}}$ is contained in $\mfc$. Further, recall the tight closure of a principal ideal agrees with its integral closure.

let $\xi$ be an element of the $R$-module $0^*_{H^1_\mfm(R)}$. By \cite[Proposition~2.5]{Smi97}, for a regular parameter $x \in R$, we have $\xi=[a+(x^n)]$ for some $a$ lying in $(x^n)^*=\overline{(x^n)}$. Thus, the element $a/x^n$ of the total ring of quotients of $R$ is inside $\overline{R}$, and for any element $b \in \mfm$, we have $b^{p^e}(a/x^n)^{p^e} \in R$ since $\mfm^{\fbp{p^e}}$ is contained in $\mfc$. Hence, $(ba)^{p^e} \in (x^{np^e})=(x^n)^{\fbp{p^e}}$, thus $ba \in (x^n)^F$. This implies $b[a+(x^n)] \in 0^F_{H^1_\mfm(R)}$ so that the $R$-module $\mfm 0^*_{H^1_\mfm(R)}$ is contained in $0^F_{H^1_\mfm(R)}$, as required.
\end{proof}

We are now prepared to prove our main result. 

\begin{sthm}\label{thm: counting branches main thm}
Let $(R,\mfm,k)$ be an excellent, reduced local ring of dimension one and of prime characteristic $p>0$. Further, let $(S,\mfn,\ell)$ be the weak normalization of $R$. Then \[\dim_k 0^*_{H^1_\mfm(R)}/0^F_{H^1_\mfm(R)} = [\ell : k](b(R)-1).\] In particular, if $k$ is perfect, \[\dim_k 0^*_{H^1_\mfm(R)}/0^F_{H^1_\mfm(R)} = b(R)-1.\] 
\end{sthm}

\begin{proof}
Since $R$ is Cohen-Macaulay, it is equidimensional and the direct limit system defining $H=H^1_\mfm(R)$ is injective. For a parameter element $x \in R$, we have $0^*_H = \varinjlim\, (x^n)^*/(x^n)$ and $0^F_H = \varinjlim\, (x^n)^F/(x^n)$. Then, $0^*_H/0^F_H = \varinjlim\, (x^n)^*/(x^n)^F$, and we show that for all $n\gg 0$, $\dim_k (x^n)^*/(x^n)^F = [\ell:k](b(R)-1)$. The fact that the direct limit system is injective then proves the claimed equality.

Since $R$ is one dimensional, the conductor ideal $\mfc$ of $R\rightarrow \overline{R}$ is $\mfm$-primary, so there is an $N$ such that $\mfm ^N$ lies in $\mfc$. Thus, by renaming $x^N$ to $x$, we may safely assume that $x$ lies in $\mfc$.

By \cite[Proposition~1.6.1]{HS} and \cite[Corollary~5.8]{HH}, we have $(x)^*=\overline{(x)} = x\overline{R} \cap R$, but $x\overline{R}\subset R$ since $x \in \mfc$, that is, $\overline{(x)} = x\overline{R}$ as $R$-submodules of $\overline{R}$. Similarly, $(x)^F = xS \cap R = xS$. Further, $x$ is a regular element on $R$ which implies $x\overline{R}/xS \simeq \overline{R}/S$, so we must compute the length of the $R$-module $\overline{R}/S$. 

Let $J\subset \overline{R}$ be the Jacobson radical of $\overline{R}$, which is the intersection of the finitely many maximal ideals $\mathfrak{M}_1,\ldots,\mathfrak{M}_b$ of $\overline{R}$. Write $K_i = \overline{R}/\mathfrak{M}_i$; by \Cref{rmk: geometric branches counted by separable degree}, we have \[\sum_{i=1}^b [K_i:k]_{\text{sep}} = b(R).\] For convenience, write $[\ell:k]=w$. Now, we have a sequence of finite field extensions $k\rightarrow \ell \rightarrow K_i$ for each $i$, and $k\rightarrow \ell$ is the perfect closure of $k$ in $K_i$. This implies $\ell \rightarrow K_i$ is a separable extension, and thus \[[K_i:k]=w[K_i:k]_{\text{sep}}\] for each $i$. 

Note that since the conductor ideal $\mathfrak c$ of $R\rightarrow \overline{R}$ has height $1$, every element of $J$ has a power which is inside $\mfc$. Thus, for some natural number $e$ we must have $F^e(J)$ is inside $R$, and so $J$ lies in $S$. Furthermore, viewed as an ideal of $S$, we must have that $J = J\cap S=\mfn$ is the unique maximal ideal of the local ring $S$ since $J$ is the Jacobson radical of $\overline{R}$ and $S\rightarrow \overline{R}$ is an integral extension.  

The map $S\rightarrow \overline{R}$ induces an injective map $\ell=S/\mfn \rightarrow \overline{R}/J$. Further, by \cite[Corollary~1.5]{LV81} we must have that the conductor $S\rightarrow \overline{R}$ is radical, and is $\mfn$-primary since $\dim S = 1$ and $S$ is reduced, so the conductor must be $\mfn$. This implies $J\overline{R} = \mfn S$ as $R$-submodules of $\overline{R}$, so we may apply the isomorphism theorems to see that the $R$-module $(\overline{R}/J)/(S/\mfn)$ is isomorphic to $\overline{R}/S$. Consequently, we get a short exact sequence of $k$-vector spaces \begin{center}
    \begin{tikzcd}
        0\arrow{r} & \ell \arrow{r} & K_1\times \ldots \times K_b \arrow{r} & \overline{R}/S \arrow{r} & 0.  
    \end{tikzcd}
\end{center} This gives the dimension equality \[\dim_k \overline{R}/S = \dim_k K_1\times \ldots \times K_b - \dim_k \ell =  \sum_{i=1}^b [K_i:k] - w, 
\] from which we get $\dim_k \overline{R}/S = w \sum_i [K_i:k]_{\text{sep}} - w = w(b(R)-1)$, as required.
\end{proof}

Together with \Cref{thm: F-nilpotent implies geo unibranched}, the following characterization of one-dimensional $F$-nilpotent rings is immediate. We invite the reader to compare it with \cite[Theorems~4.12,~4.16]{Bli04}.

\begin{scor}\label{col:main corollary}
Let $R$ be a locally excellent, reduced ring of dimension one. Then, $R$ is $F$-nilpotent if and only if $R_\mfm$ is geometrically unibranched for each maximal ideal $\mfm$ of $R$.
\end{scor}

The arguments involved in the proof of \Cref{thm: counting branches main thm} provide the following technique to compute the number of geometric branches using ideals instead of submodules of local cohomology. 

\begin{scor}\label{cor: counting dimension for tight mod frob}
Let $(R, \mfm, k)$ be an excellent, reduced local ring of dimension one and let $(S, \mfn, \ell)$ be the weak normalization of $R$. If $x$ is a regular element in the conductor ideal of $R$, then \[\dim_k\, (x)^*/(x)^F = [\ell:k](b(R)-1).\]
\end{scor}

\begin{srmk}\label{rmk: SW reconciliation}
Suppose $(R,\mfm)$ is a complete local ring with an algebraically closed coefficient field. In \cite{SW08}, Singh-Walther generalized a result of Lyubeznik to count the number of connected components of the punctured spectrum of $R^{sh}$. In particular, if $\dim R = 1$, their formula counts the number of geometric branches of $R$. 

The formula of Singh-Walther uses the \textit{semi-stable part} $H_{\text{ss}}$ of $H=H^1_\mfm(R)$ whose definition is given below. 

\[H_{\text{ss}} = \text{Span}_k\, \bigcap_e F^e(H)\]

Clearly, $H_{\text{ss}}$ is an $F$-stable $R$-submodule of $H$. Singh-Walther show under the hypotheses above that $\dim_k H_{\text{ss}} = b(R)-1$,  just as we show $\dim_k 0^*_{H}/0^F_{H} = b(R)-1$. Thus, the $k$-vector spaces $H_{\text{ss}}$ and $0^*_H/0^F_H$ are isomorphic in this setting. In general, computing the space $H_{\text{ss}}$ can be quite difficult. On the other hand, we give an explicit formula to compute a $k$-vector space basis of $0^*_H/0^F_H$ when $R$ is standard graded over a perfect field $k$; see \Cref{lem: tight and frob closure of reduction of m}. Below, we show the isomorphism $H_{\text{ss}}\simeq 0^*_{H}/0^F_{H}$ is due to a splitting $0^*_{H} \simeq 0^F_{H} \oplus H_{\text{ss}}$ so long as $0^*_{H^d_\mfm(R)}$ is finite length.
\end{srmk}

We now reconcile our main result with that of Singh-Walther; for this, we will need the following splitting of vector spaces with a Frobenius action.

\begin{sprop}\cite[\S 1.3]{ST17} \label{prop: fd vector space splits}
Let $k$ be a perfect field and let $V$ be a finite-dimensional $k$-vector space equipped with a $p$-linear map $f:V\rightarrow V$, that is, an additive map such that $f(\lambda x)=\lambda^pf(x)$ for all $\lambda \in k$ and $x \in V$. Then, $V\simeq V_{\text{ss}}\oplus V_{\text{nil}}$, where $V_{\text{ss}}$ is defined as above and $V_{\text{nil}} = \cup_e \ker(f^e)$. \end{sprop} 

\begin{sthm}
Let $k$ be an algebraically closed field and let $(R,\mfm,k)$ be an excellent, equidimensional, reduced ring of dimension $d>0$ which contains $k$ as a coefficient field. Further, suppose $0^*_{H^d_\mfm(R)}$ is finite length over $R$ (for instance, if $R$ is $F$-rational on the punctured spectrum or reduced and of dimension 1). Then, we have the $k$ vector space isomorphism: \[ 0^*_{H^d_\mfm(R)}/0^F_{H^d_\mfm(R)}\simeq H^d_\mfm(R)_{\text{ss}}.
\]
\end{sthm}

\begin{proof}
Note that since $\operatorname{depth} R \ge 1$, $\mfm$ is not an associated prime so that $\mfm^N \cap R^\circ$ is nonempty for all $N$. Now, Write $H=H^d_\mfm(R)$. 

By \cite[Prop.~5]{Dieu} and \cite[Thm.~1.12]{HaSp}, since $k$ is an algebraically closed coefficient field, $H_{\text{ss}}$ is a finite dimensional $k$-vector space, and there is a basis $\{\xi_1,\cdots,\xi_t\}$ of $H_{\text{ss}}$ such that $F(\xi_j)=\xi_j$ for each $j$. Furthermore, since $H$ is $\mfm$-torsion, there is a natural number $N$ and an element $c$ in $R^\circ$ such that $c \in \mfm^N$ and $c\xi = 0$ for all $\xi \in H_{\text{ss}}$. Thus, $cF^e(\xi_j)=c\xi_j=0$ for all $j$ and natural numbers $e$, that is, $H_{\text{ss}}\subset 0^*_H$. Furthermore, the inclusion $0^*_H\subset H$ implies $(0^*_H)_{\text{ss}}\subset H_{\text{ss}}$, so we have $(0^*_H)_{\text{ss}}=H_{\text{ss}}$. Similarly, since $0^F_H \subset 0^*_H$, we have $0^F_H = (0^*_H)_{\text{nil}}$.

As $k$ is a coefficient field, $0^*_H$ is a finite dimensional $k$-vector space. Then, by \Cref{prop: fd vector space splits}, we have \[0^*_H \simeq (0^*_H)_{\text{ss}} \oplus (0^*_H)_{\text{nil}} = H_{\text{ss}}\oplus 0^F_H,\] which gives the required isomorphism.

In the case that $R$ is $F$-rational on the punctured spectrum, by \cite[Lemma~2.3]{ST17}, we must have that $0^*_{H^d_\mfm(R)}$ is finite length. Finally, if $\dim R = 1$ and $R$ is reduced, then it is a field on the punctured spectrum.
\end{proof}

If $d=1$, the theorem above specializes to a proof that our count of geometric branches in \Cref{thm: counting branches main thm} agrees with the main result of Singh-Walther, however if $d>1$, the two seem unrelated.  

Using our techniques, we are able to prove following result which is likely well-known to experts and is independent of the characteristic of the field; it follows from the projective Nullstellensatz. In contrast, it does not seem to be easily recoverable from the main result of Singh-Walther (\cite{SW08}) in dimension one. 

\begin{scor}\label{cor: hs multiplicity branches}
Let $k$ be an infinite perfect field of positive characteristic. A one dimensional reduced, standard graded algebra over $k$ with Hilbert-Samuel multiplicity $e(R)$ has $e(R)$ geometric branches.
\end{scor}

\begin{proof}
Let $R$ be a reduced, standard graded $k$-algebra with homogeneous maximal ideal $\mfm$. For all $n\gg 0$, we have $e(R)=\dim_k \mfm^n /\mfm^{n+1}$. Let $x\in [R]_1$ be a linear form which reduces $\mfm$, then for all $n$, $x^n$ is a part of a minimal generating set of the ideal $\mfm^n$. Let $J=(x^n)+\mfm^{n+1}$; we have a short exact sequence of $k$-vector spaces: \begin{center}
    \begin{tikzcd}
        0\arrow{r} & J/\mfm^{n+1} \arrow{r} & \mfm^n/\mfm^{n+1} \arrow{r} & \mfm^n/J \arrow{r} & 0.
    \end{tikzcd}
\end{center} 

Hence, $e(R)-1 = \dim_k \mfm^n/\mfm^{n+1} - 1  = \dim_k \mfm^n/J$. But then by \Cref{lem: tight and frob closure of reduction of m} and \Cref{thm: counting branches main thm}, we know $\mfm^n/J =(x^n)^*/(x^n)^F$, so we have $b(R)=e(R)$. 
\end{proof}

We conclude this section by using \Cref{thm: counting branches main thm}, \Cref{cor: counting dimension for tight mod frob}, and \Cref{cor: hs multiplicity branches} to count the number of geometric branches of some reduced, one dimensional local rings. 

\begin{sex} \label{ex1: dimension count example 1}
Let $k$ be a field of characteristic $p>0$ and let $S=k[x_1,\ldots,x_d]$ and $I=(x_ix_j \mid i<j)$, with $R=S/I$ localized at the maximal ideal $\mfm = (x_1, \ldots, x_d)$. 

Notice that $x=x_1+\ldots+x_d$ is a parameter element of $R$, and for any $n \ge 1$ we have $x^n = x_1^n + \ldots + x_d^n$ and $\mfm^n = (x_1^n,\ldots,x_d^n)$.

Since $R$ is defined by squarefree monomials, it is $F$-pure, so that $0^F_{H^1_\mfm(R)} = 0$. Therefore, we only need to compute the $k$-vector space dimension of  $0^*_{H^1_\mfm(R)}$. As $0^*_{H^1_\mfm(R)}=\varinjlim (x^n)^*/(x^n)$, we begin by computing $(x^n)^*$.

We immediately see that $x \in R^\circ$, and $x \cdot x_i^{np^e} = x_i^{np^e+1} = x_i \cdot x^{np^e}$, so that $x_i^n \in (x^n)^*$ for each $i$. Consequently, the ideal $\mfm^n$ is contained in $(x^n)^*$ and by degree considerations, we must have $(x^n)^* = \mfm^n$. Then, \[0^*_{H^1_\mfm(R)} = \varinjlim \mfm^n/(x^n) = \varinjlim (x^n,x_2^n,\ldots,x_d^n)/(x^n).\] Note that the set $B:=\{[x_i+(x)]\mid 2 \le i \le d\}$ is a $k$-basis of $0^*_{H^1_\mfm(R)}$ as the direct limit system defining $H^1_\mfm(R)$ is injective. Consequently, $\dim_k\, 0^*_{H^1_\mfm(R)}/0^F_{H^1_\mfm(R)} = d-1$; It follows that $R$ has $d$ geometric branches by \Cref{thm: counting branches main thm}.
\end{sex}

\begin{sex}
Let $k$ be a perfect field of characteristic $p>0$ and let $S$ be the polynomial ring $k[x,y]$ with the standard grading and $\mathfrak m = (x,y)S$. Pick $f \in S$ a homogeneous form of degree $d$ such that $R=S/fS$ is reduced. Since $f$ lies in $\mathfrak m^d$ but not in $\mathfrak m^{d+1}$, the Hilbert-Samuel multiplicity of $R_\mfm$ is $d$ (see \cite[11.2.8]{HS}). By \Cref{cor: hs multiplicity branches}, we have that $R_{\mathfrak m}$ has $d$ geometric branches. 
\end{sex}  

We also demonstrate a calculation of the number of geometric branches for ring of dimension one which is not graded.

\begin{sex}
Let $k$ be a field of characteristic $p>0$ and let $R=k[x,y]/(y^2-x^3-x^2)$ (sometimes called the nodal cubic curve) localized at the ideal $(x,y)$. The element $x \in R$ is a parameter in the conductor of $R\rightarrow \overline{R}$, so we may apply \Cref{cor: counting dimension for tight mod frob}. Clearly $(x)^*=\overline{(x)} = (x,y)$; the remainder of the computation depends on the characteristic of $k$.

If $p = 2$, note that $y^2=x^3+x^2=x^2(x+1)$, so $y^2 \in (x^2)R$ which implies $(x)^F = (x,y)$. Thus, $\dim_k (x)^*/(x)^F = 0$ so that $b(R)=1$. If $p$ is odd, then by Fedder's criterion \cite[Propositon~2.1]{Fed}, $R$ is $F$-pure, and so $(x)^F=(x)$. Thus, $\dim_k (x)^*/(x)^F = 1$ and $b(R)=2$. 
\end{sex}

The above calculation agrees with the fact that in any characteristic other than $2$, the completion of the nodal cubic curve at $(x,y)$ allows us to factor 
\[y^2 -x^3-x^2 = (y-x\sqrt{1+x})(y+x\sqrt{1+x}).\] By \Cref{col:main corollary}, we get that the ring $k[x,y]/(y^2-x^3-x^2)$ is $F$-nilpotent if and only if the characteristic of $k$ is $2$ since it has an isolated singularity at $(x,y)$. 

\begin{srmk}
The property of being geometrically unibranched does not characterize $F$-nilpotence in general. Any normal local domain with a separably closed residue field must be geometrically unibranched but need not be $F$-nilpotent. For a particular example, let $k$ be a separably closed field of prime characteristic $p$ and let $R=k[|x,y,z|]/(x^2+y^3+z^7+xyz)$. Then $R$ is an $F$-injective normal local domain but is not $F$-nilpotent (as it is not $F$-rational). 

Further note that since $R$ is not $F$-nilpotent, the modules $0^*_{H^2_\mfm(R)}$ and $0^F_{H^2_\mfm(R)}$ are not equal while $R$ is geometrically unibranched. So, the formula to count the number of geometric branches does not immediately extend to higher dimensions. It would be interesting to find an extension of our formula for local rings of higher dimensions.
\end{srmk}

\section{Computational aspects of purely inseparable extensions}\label{sec: Fte and pi exts}

In this section, we work towards using purely inseparable extensions to compute tight and Frobenius closure of ideals. First, we record a theorem regarding the Frobenius test exponent of an ideal extended or contracted along a purely inseparable extension.

\begin{sthm}\label{thm: computing fte along pi exts}
Let $R\rightarrow S$ be a purely inseparable ring extension with finite pure inseparability index $e_0$, and let $I$ and $J$ be ideals of $R$ and $S$ respectively. Then, $\fte I\le \fte IS +e_0$ and $\fte J\le \fte F^{e_0}(J)R +e_0$. In particular, if $(R,\mfm)$ and $(S,\mfn)$ are local, then $$\fte R \le \fte S+e_0 \text{ and } \fte S \le \fte R + e_0.$$
\end{sthm}

\begin{proof}
Write $I=(x_1,\ldots,x_n)R$. Then, for any $x \in I^F$, we have $x \in I^FS \subset (IS)^F$, so for $e = \fte IS$, we have an equation $x^{p^e} = \sum s_i x_i^{p^e}$ for some $s_i \in S$. Further, since $F^{e_0}(s_i)\in R$, we get $x^{p^{e+e_0}} = \sum s_i^{p^{e_0}} x_i^{p^{e+e_0}}$ is a Frobenius closure equation in $R$, thus $\fte I \le e+e_0 = \fte IS + e_0$.

Similarly, write $J=(y_1,\ldots,y_m)S$. Then, for any $y \in J^F$ we have an $e \in \mbn$ and $s_1,\ldots,s_m$ in $S$ with $y^{p^e} = \sum s_i y_i^{p^e}$. By applying $F^{e_0}$ we get $y^{p^{e+e_0}} = \sum s_i^{p^{e_0}} y_i^{p^{e+e_0}}$ is a Frobenius closure equation demonstrating $y^{p^{e_0}} \in ((y_1^{p^{e_0}},\ldots,y_m^{p^{e_0}}) R)^F$. Consequently, we can choose $e=\fte (y_1^{p^{e_0}},\ldots,y_m^{p^{e_0}}) R + e_0$ in our initial equation showing $y\in J^F$ in $S$ independent of $y$.

In the local case, if $\mfq \subset R$ is a parameter ideal, $\mfq S$ is also a parameter ideal of $S$, so $\fte R \le \fte S + e_0$ is shown. Similarly, if $\mathfrak{s}$ is a parameter ideal of $S$, then $F^{e_0}(\mathfrak{s})R$ is a parameter ideal of $R$, so $\fte S \le \fte R + e_0$. 
\end{proof}

It is an interesting and difficult open problem to determine all rings which have finite Frobenius test exponent. From the above theorem, we see that it suffices to consider weakly normal rings when attempting to answer this question.

\begin{scor}\label{cor: weak normalization finite fte}
Let $R$ be an excellent, reduced local ring of prime characteristic and write $S$ for its weak normalization. Then, $\fte R$ is finite if and only if $\fte S$ is finite. \end{scor}

We now turn to purely inseparable extensions and tight closure. 

\begin{srmk}\label{rmk: pi exts and R circ} Let $\phi:R\rightarrow S$ be a purely inseparable extension. Notice for each minimal prime $\mfp$ of $S$, $\phi^{-1}(\mfp)$ is a minimal prime of $R$ since $\phi$ induces a homeomorphism between $\spec S$ and $\spec R$. Therefore the set $\phi(R^\circ)$ is contained in $S^\circ$ and for each $x \in S^\circ$ there is an $e \in \mbn$ such that $x^{p^e} \in R^\circ$. \end{srmk}

Heuristically, if $R\rightarrow S$ is purely inseparable, we should expect a nilpotent version of the singularity type of $S$ to descend to $R$. In the context of ideal closures, closure properties which hold in $S$ should also hold in $R$ up to Frobenius closure. The following theorem demonstrates one example of this principle, with an application to affine semigroup rings.

\begin{sthm}\label{thm: sf-nilp if pi ext is freg}
Let $\phi:R\rightarrow S$ be a finite, purely inseparable extension with pure inseparability index $e_0$, and suppose that $S$ is weakly $F$-regular. Then, for any ideal $I$ of $R$, we have $I^*=I^F$ and $\fte I \le e_0$.
\end{sthm}

\begin{proof}
Let $I\subset R$ be an ideal and suppose $I=(x_1,\ldots,x_n)R$. Then, $I^*S\subset (IS)^*=IS$ since $\phi(R^\circ)\subset S^\circ$, so for any $x \in I^*$, we have $x=\sum s_i x_i$ for some $s_1,\ldots,s_n \in S$. Then, $x^{p^{e_0}} = \sum s_i^{p^{e_0}} x_i^{p^{e_0}}$ is a Frobenius closure equation in $R$, so $x \in I^F$. Further, $e_0$ is independent of $x$, so $\fte I\le e_0$.
\end{proof}

\begin{srmk}
A variety of modifications to \Cref{thm: sf-nilp if pi ext is freg} can be made by replacing the requirement that $S$ be weakly $F$-regular with another condition defined in terms of ideal closures. For instance, if $R\rightarrow S$ is finite and purely inseparable with pure inseparability index $e_0$, and $S$ is $F$-pure,  then $\fte I \le e_0$ for all ideals $I$ of $R$.
\end{srmk}

The above theorem shows that tight and Frobenius closure are more easily computable in rings where a known strongly $F$-regular purely inseparable extension exists and the pure inseparability index can be calculated. In particular, this gives us a method to compute the tight closure of an ideal in an $F$-nilpotent affine semigroup ring defined over a field. For us, an affine semigroup $A$ is a finitely generated sub-monoid of $\mbn^n$ for some $n$.

\begin{scor}\label{cor: f-nilp affine semigroup ring}
Suppose $R$ is a locally excellent domain and its integral closure $\overline{R}$ is $F$-regular (for instance, if $R$ is an affine semigroup ring defined over a field $k$ of prime characteristic $p>0$). Then, $R$ is $F$-nilpotent if and only if $R\rightarrow \overline{R}$ is purely inseparable. Further, if this is the case, then  $I^F=I^*$ for all ideals $I$ of $R$ and $\fte I\le e_0$ where $e_0$ is the pure inseparability index of $R\rightarrow \overline{R}$. 
\end{scor}

\begin{proof}
If $R$ is $F$-nilpotent, the proof of \Cref{thm: F-nilpotent implies geo unibranched} implies that $R\rightarrow \overline{R}$ is purely inseparable. Now suppose $\iota: R\rightarrow \overline{R}$ is purely inseparable, and notably $\iota$ is finite since $R$ is an excellent domain. 

Since $\iota$ is purely inseparable, it induces a homeomorphism on spectra. In particular, for all primes $\mfp \in \spec R$, there is a unique prime $\mfq$ of $\overline{R}$ lying over $\mfp$, and in particular, the same is true for maximal ideals $\mfm$ of $R$. Consequently, we may replace $R$ and $\overline{R}$ with localizations at a maximal ideal to assume $\iota: (R,\mfm) \rightarrow (\overline{R},\mfn)$ is the normalization map of the local ring $R$, and $\iota$ is finite and purely inseparable. 

Now, since $\overline{R}$ is an $F$-regular ring, it is $F$-nilpotent, and so \Cref{thm: ascent/descent of F-nilpotence} implies $R$ is $F$-nilpotent as well. We may now apply \Cref{thm: sf-nilp if pi ext is freg} to see the final claim, since $F$-regular rings are weakly $F$-regular. 

Finally, if $A$ is an affine semigroup and $R=k[A]$ is the associated affine semigroup ring, then $\overline{R}$ is a direct summand of a polynomial ring and is thus $F$-regular.
\end{proof}

Note that in the proof above, we only need the weaker assumption that $\overline{R}$ is $F$-nilpotent (and not necessarily $F$-regular) to conclude that $R$ is $F$-nilpotent if and only if $R\rightarrow \overline{R}$ is purely inseparable. Furthermore, we note that the rings described above avoid the problem raised by Brenner in \cite[Theorem~2.4]{Bre06}, that is, they have a uniform trivializing exponent for Frobenius closure over all ideals simultaneously. 

In \cite{MP22}, the latter two authors of this paper studied a family of examples of affine semigroup rings called \textit{pinched Veronese rings}, formed by removing a single algebra generator from a Veronese subring of a polynomial ring. All but one small family of examples of pinched Veronese rings are $F$-nilpotent, and for these rings, the number $e_0$ in the corollary above is $1$ (see \cite[Theorem B, Corollary 4.9]{MP22}). This vastly improves the previously known bounds for the Frobenius test exponents of these rings $R$, which were roughly of the order $\fte R \le 2^{\dim R}$ by \cite[Theorem~4.2]{Quy}. 

\begin{sex}
Let $R = k[x^2, xy,xz,y^2,z^2]$; note that the integral closure $\overline{R}$ of $R$ is the Veronese subring $k[x,y,z]^{(2)}$ of the polynomial ring $k[x,y,z]$ under the standard grading.

In \cite[Theorem B]{MP22}, the latter two authors of this paper showed that the pure insperability of the normalization map $\iota: R\rightarrow \overline{R}$ depends on the characteristic $p$ of $k$. In particular, $\iota$ is purely inseparable if and only if $p=2$. Further, if $p=2$, then the pure inseparability index of $\iota$ is $1$. Thus, by \Cref{cor: f-nilp affine semigroup ring}, $R$ is $F$-nilpotent if and only if $p=2$, and in this case, for any ideal $I$ of $R$, we have $x \in I^*$ if and only if $x^p \in I^{\fbp{p}}$.

Finally, if $p$ is odd, then $R$ is in fact \textit{$F$-pure}, so all ideals of $R$ are Frobenius closed.  
\end{sex}

\begin{srmk}\label{rmk: f-coherent}
An additional class of rings $R$ with the property that $I^*=I^F$ for all ideals $I$ of $R$ are \textit{$F$-coherent} rings, whose definition we avoid here. The class of $F$-coherent rings were studied by Shimomoto in \cite{Shi11}. We note the property that $I^*=I^F$ for all ideals $I$ does not characterize $F$-coherent rings, since there are $F$-regular rings (in which $I^*=I^F=I$ for all ideals $I$) which are not $F$-coherent, see \cite[Example~3.14]{Shi11}.
\end{srmk}

In general, $F$-coherence is a much stronger property than $F$-nilpotence. However, for rings of dimension one, they are equivalent.

\begin{sprop}\label{prop: f-coh iff f-nilp in dim 1}
Let $R$ be an excellent, reduced local ring of dimension one. Then, $R$ is $F$-coherent if and only if $R$ is $F$-nilpotent.
\end{sprop}

\begin{proof}
By \cite[Corollary~3.8]{Shi11}, $R$ is $F$-coherent if and only if the normalization map $R\rightarrow \overline{R}$ is purely inseparable. In particular, if $R$ is $F$-coherent, then $R$ is geometrically unibranched, which implies that $R$ is $F$-nilpotent by \Cref{col:main corollary}. Conversely, if $R$ is $F$-nilpotent, then $R\rightarrow \overline{R}$ is purely inseparable by \Cref{thm: F-nilpotent implies geo unibranched}, so $R$ is $F$-coherent.
\end{proof}

\section*{Acknowledgements}
The authors would like to thank Linquan Ma, Anurag Singh, and Uli Walther for several helpful discussions. We are also grateful to Austyn Simpson for discussions related to \Cref{rmk: f-coherent} and \Cref{prop: f-coh iff f-nilp in dim 1}. Finally, we would like to thank the anonymous referee for their comments which improved the exposition of the article.

\bibliographystyle{alpha}
\bibliography{ref}

\end{document}